\documentclass[psamsfonts]{amsart}

\usepackage{amssymb,amsfonts}
\usepackage[all,arc]{xy}
\usepackage{enumerate}
\usepackage{pgf}
\usepackage{pgfkeys}
\usepackage{tikz}
\usetikzlibrary{arrows}
\usetikzlibrary{decorations.markings}
\usepackage{chngcntr}
\counterwithin{figure}{section}
\usetikzlibrary{calc}
\usepackage{xifthen}

\newtheorem{thm}{Theorem}[section]
\newtheorem{cor}[thm]{Corollary}

\newtheorem{lem}[thm]{Lemma}

\theoremstyle{definition}
\newtheorem{defn}[thm]{Definition}
\newtheorem{defns}[thm]{Definitions}
\newtheorem{con}[thm]{Construction}

\newtheorem{notn}[thm]{Notation}

\newtheorem{cond}[thm]{Condition}

\theoremstyle{remark}
\newtheorem{rem}[thm]{Remark}

\newcommand{\dz}{\mathrm{d}z}

\newcommand{\nc}{\newcommand}
\nc{\kls}{\mathrm{LabSur}_k}
\nc{\slr}{\mathrm{SL}(2,\mathbb R)}
\nc{\R}{\mathbb R}
\nc{\C}{\mathbb C}
\nc{\Z}{\mathbb Z}
\nc{\N}{\mathbb N}
\nc{\td}{\sim}
\nc{\mc}{\mathcal}
\nc{\ec}{\mc{EC}(T,x)}
\nc{\ov}{\overline}
\nc{\Arg}{\mathrm{Arg}}
\nc{\Hom}{\mathrm{Hom}}
\nc{\dbz}{\mathrm{d}\bar z}
\nc{\vp}{\varphi}
\renewcommand{\d}{\mathrm{d}}
\pgfmathsetmacro{\octagonradius}{3}
\pgfmathsetmacro{\octagonbigradius}{\octagonradius/sin(67.5)}

\makeatletter
\let\c@equation\c@thm
\tikzset{nomorepostaction/.code={\let\tikz@postactions\pgfutil@empty}}
\g@addto@macro\@floatboxreset\centering
\makeatother
\numberwithin{equation}{section}

\title{Periodic Billiards in Isosceles Triangles}

\author{Alex Becker}

\begin{document}

\begin{abstract}
Any periodic trajectory on an isosceles triangle gives rise to a periodic 
trajectory on a right triangle obtained by identifying the halves of the original triangle. We examine the relationship between periodic trajectories on isosceles triangles and
the trajectories on right triangles obtained in this manner, 
and the consequences of this relationship for the existence of stable trajectories on isosceles triangles and the properties of
their orbit tiles.
\end{abstract}

\maketitle

\tableofcontents

\section{Introduction}

We consider an idealized model of a billiard ball on a convex polygonal table, behaving as a point mass with no friction which
collides perfectly elastically with the edges of the table.
Such systems are referred to as ``polygonal billiards" and are a common subject
of study in the field of dynamics. A complete description of the system can be found in a survey of the subject by 
Boldrighini, Keane and Marchetti \cite{Billiards}.
%

Many natural questions about polygonal billiards remain open. 
Most notably, it is unknown whether general polygonal billiards admit periodic trajectories, and whether generic polygonal billiards are ergodic.
An account of the proven and conjectured properties of periodic trajectories can be found in
Section 1 of Schwartz's article \cite{Deg100}.
Restricted cases have been shown to have periodic trajectories, including 
rational polygons as proven by Masur in \cite{RatCase}, and triangles with maximal angles less than $100$ degrees as proven by
Schwartz in \cite{Deg100}. Discussion of ergodicity results can be found in a survey by Gutkin \cite{Survey}. 
Recently there has been considerable interest in two closely related objects associated with periodic trajectories, the
\emph{combinatorial type} and the \emph{orbit tile}. 

\begin{defn}
Let $f$ be a periodic trajectory. The \emph{combinatorial type} of $f$ is the sequence $\sigma_1\cdots\sigma_k$ of edges struck by
$f$. Two combinatorial types are considered equivalent if they are identical or one is obtained by reversing the other.
\end{defn}

\begin{defn}
Given a combinatorial type $C$, the \emph{orbit tile} $\mathcal O(C)$ of $C$ is the set of polygons which admit a periodic trajectory with combinatorial type $C$.
The orbit tile of a periodic trajectory is the orbit tile of its combinatorial type.
\end{defn}

Analysis of these objects has been successfully applied to the problem of finding periodic trajectories,
most notably by Hooper and Schwartz in \cite{NearlyIsosc, Deg1001, Deg100}. This work has suggested that \emph{stable}
periodic trajectories are of particular interest (see for example \cite{RightTri, NearlyIsosc}).

\begin{defn}
A periodic trajectory is \emph{stable} if its orbit tile is open. 
\end{defn}

This paper focuses exclusively on the theory of periodic trajectories up to combinatorial type; we let $\sim$ denote this equivalence.
We restrict our attention to the special case of an isosceles triangle $T$. Hooper and Schwartz showed in \cite{NearlyIsosc}
that all isosceles triangles---and in fact triangles sufficiently close to being isosceles---possess periodic trajectories.
One interesting property of periodic trajectories $f$ on $T$ is that, by identifying each point in $T$ with its
reflection across the line of symmetry, we obtain a periodic trajectory $f'$ on the resulting right triangle $T'$.
This paper studies the relationship between these trajectories, and in particular between their combinatorial types.
This relationship turns out to be governed by whether $T$ satisfies the following condition.
 
\begin{cond}\label{condition}
The base angle of $T$ is either an irrational multiple of $\pi$ or of the form $\frac{a}{b}\pi$ with $a,b\in\Z$ and $b$ even.
\end{cond}

In particular, we prove the following theorem.

\begin{thm}\label{main theorem}
Let $f$ be a periodic trajectory on an isosceles triangle $T$ satisfying Condition ~\ref{condition}.
If $g$ is a periodic trajectory on $T$ and $f'\sim g'$, then $f\sim g$.
\end{thm}

Continuing in the tradition of Hooper and Schwartz, we interpret this result as a statement about orbit tiles.
As usual, we regard the set of triangles as the set of pairs $(\theta_1,\theta_2)$ with $0<\theta_1,\theta_2<\frac{\pi}{2},$
which are interpreted as triangles with angles $\theta_1,\theta_2$ and $\pi-\theta_1-\theta_2$.

\begin{cor}\label{corollary}
Any orbit tile which contains an isosceles triangle satisfying Condition ~\ref{condition} is symmetric across the line of isosceles triangles.
\end{cor}

For isosceles triangles not satisfying Condition ~\ref{condition}, Theorem ~\ref{main theorem} fails drastically.
In fact, we prove:

\begin{thm}\label{converse}
Let $T$ be an isosceles triangle with base angle $\frac{a}{b}\pi$, $b$ odd, $x\in T$ a point in the base other than its center,
$\theta\in S^1$ and $\mathrm{Orb}(x,\theta)$ the trajectory originating from $x$ with direction $\theta$.
Then the set of $\theta$ such $\mathrm{Orb}(x,\theta)$ is periodic and satisfies
Theorem ~\ref{main theorem} is nowhere dense.
\end{thm}

By a well-known result of Masur in \cite{RatCase}, the set of $\theta$ such that the trajectory originating from $x$ with direction
$\theta$ is periodic is dense. Thus Theorem ~\ref{converse} provides a converse to Theorem ~\ref{main theorem}. Because
Theorem ~\ref{main theorem} holds on a dense set of isosceles triangles, Theorem ~\ref{converse} can be interpreted as a result
on stable trajectories.

\begin{cor}\label{unstable}
Let $T$ be an isosceles triangle with base angle $\frac{a}{b}\pi$, $b$ odd, and $x\in T$ a point in the base other than its center.
Then the set of $\theta$ such that $\mathrm{Orb}(x,\theta)$ is periodic and stable is nowhere dense.
\end{cor}

\section{Preliminaries}

In this section we work with arbitrary convex polygons rather than restricting our attention to triangles.
Let $P\subset \mathbb C$ be a convex $n$-gon with vertices $v_i$, labeled counter-clockwise.
Let $e_i$ denote the edge from $v_i$ to $v_{i+1}$, with addition interpreted mod $n$.

\begin{defns}
A \emph{periodic trajectory} is a function $f:[0,1]\to P$ with constant unit derivative, except at the edges of $P$ where the derivative
is reflected across the edge.
A \emph{cylinder of periodic trajectories} is a continuous function $g:(0,1)\times [0,1] \to P$ such that fixing the first coordinate
 of $g$ for any $s\in (0,1)$ gives a periodic trajectory $g_s$, and such that $\left.\frac{dg_s}{dt}\right|_{0}$ is independent of $s$.
\end{defns}

Recall that periodic trajectories give rise to lines in the \emph{unfolding} of $P$, where instead of reflecting the billiard ball off
an edge, we relfect the polygon across the edge, as in Figure 2.1.

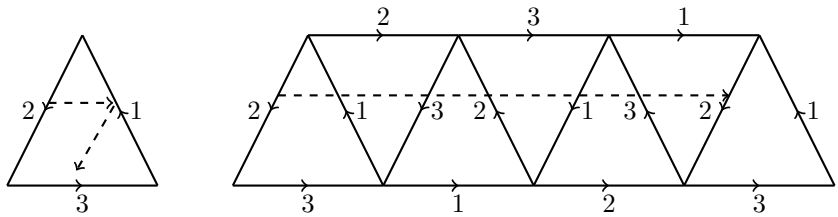
\begin{figure}
\begin{tikzpicture}[scale=2]
\begin{scope}[ 
    every path/.style={
        thick,
        postaction={nomorepostaction,decorate,
                    decoration={markings,mark=at position 0.5 with {\arrow{>}}}
                   }
        }
    ]
\draw (-3+-.5,-.5) -- (-3+.5,-.5) node[midway, below] {3};
\draw (-3+.5,-.5) -- (-3+0,.5) node[midway, right] {1};
\draw (-3+0,.5) -- (-3+-.5,-.5) node[midway, left] {2};
\draw (-2,-.5) -- (-1,-.5) node[midway, below] {3};
\draw (-1,-.5) -- (-1.5,.5)  node[midway, right] {1};
\draw (-1.5,.5) -- (-2,-.5) node[midway, left] {2};
\draw (-1.5,.5) -- (-.5,.5) node[midway, above] {2};
\draw (-.5,.5) -- (-1,-.5) node[midway, right] {3};
\draw (-1,-.5) -- (0,-.5) node[midway, below] {1};
\draw (0,-.5) -- (-.5,.5) node[midway, left] {2};
\draw (-.5,.5) -- (.5,.5) node[midway, above] {3};
\draw (.5,.5) -- (0,-.5) node[midway, right] {1};
\draw (0,-.5) -- (1,-.5) node[midway, below] {2};
\draw (1,-.5) -- (.5,.5) node[midway, left] {3};
\draw (.5,.5) -- (1.5,.5) node[midway, above] {1};
\draw (1.5,.5) -- (1,-.5) node[midway, left] {2};
\draw (1,-.5) -- (2,-.5)  node[midway, below] {3};
\draw (2,-.5) -- (1.5,.5) node[midway, right] {1};
\end{scope}
\draw[dashed, thick, ->] (-3+-.225,.05) -- (-3+.205,.05);
\draw[dashed, thick, ->] (-3+.21,.03) -- (-3+-.04,-.4);
\draw[dashed, thick, ->] (-1.7,.1) -- (1.3,.1);
\end{tikzpicture}
\caption{On the left is a trajectory on an isosceles triangle. On the right is the unfolding along that trajectory.
The edges of the triangle are labeled an oriented.}
\end{figure}

We will make heavy use of the unfolding, particularly in the rational case where
it forms a compact Riemann surface (in fact, a translation surface). We develop the unfolding using the following notation.

\begin{notn}
Let $r_i$ denote the (real) linear part of the reflection map across $e_i$. Let $G(P)$ be the group of such maps.
\end{notn}

Recall that $G(P)$ is finite iff $P$ is rational, in which case we can make use of the following construction.
This construction is essentially identical to the one found in the survey by Masur and Tabachnikov \cite{FlatStruc}.

\begin{con}\label{construction}
We construct the Riemann surface $R(P)$ by gluing together the polygons $\alpha P$ for $\alpha\in G(P)$.
We glue $\alpha P$ and $r_i\alpha P$ along $e_i$ such that $v_i\in \alpha P$ is identified with $v_i\in r_i\alpha P$
and $v_{i+1}\in \alpha P$ with $v_{i+1}\in r_i\alpha P$. The inclusion maps $\vp_\alpha:\alpha P\hookrightarrow R(P)$ give
us local parametrizations near every point in the interior of $\alpha P$. For points in the edge $\vp_\alpha(\alpha e_i)$, we define
local parametrizations via the inclusion map $\alpha P\cup r_i\alpha P\hookrightarrow R(P)$. 
At the vertex $\vp_\alpha(\alpha v_i)$ the polygons
$\alpha P, r_{i-1}\alpha P, r_i r_{i-1}\alpha P,r_{i-1}r_ir_{i-1}\alpha P,\ldots, (r_ir_{i-1})^m\alpha P = \alpha P$ meet.
We define a local parametrization near $\vp_\alpha(\alpha v_i)$ by
\[\vp(z) = \begin{cases}
\vp_\alpha(z^k) &\text{if } 0\leq \Arg(z)\leq \frac{\pi}{m}\\
\vp_{r_{i-1}\alpha}(z^k) &\text{if } \frac{\pi}{m}\leq \Arg(z)\leq \cdot\frac{2\pi}{m}\\
\quad\vdots&\quad\quad\quad\quad\ \vdots\\
\vp_{r_{i-1}(r_ir_{i-1})^{m-1}\alpha}(z^k) &\text{if } \frac{(2m-1)\pi}{m}\leq \Arg(z)\leq \cdot\frac{2m\pi}{m}
\end{cases}\]
where $\phi=\frac{k}{m}\pi$ is the angle between $e_{i-1}$ and $e_i$. Together these maps form an atlas for $R(P)$.

The standard $1$-form $\dz$ on $\C$ induces a holomorphic $1$-form $\dz_P$ via the inclusions $\vp_\alpha$. It is easy to see that
this is nonzero except at the vertices $\vp_\alpha(\alpha v_i)$, where it has a zero of order $k-1$.

We have a natural action of $G(P)$ on $R(P)$ where $\beta\in G(P)$ sends $\vp_\alpha(\alpha P)$ to $\vp_{\beta\alpha}(\beta\alpha P)$.
\end{con}

Illustrations of Construction ~\ref{construction} can be found in Figures 3.2 and 3.3.
While it is common to square $\dz_P$ to produce a quadratic differential, we shall restrict our attention to the $1$-form.

\begin{defns}
Trajectories on $R(P)$ are paths $f$ which do not contain any vertex such that $\dz_P\circ \d f = e^{i\theta}$ for some fixed $\theta,$
called the \emph{argument} of $f$. A cylinder of periodic trajectories is defined similarly.
\end{defns}

The requirement that trajectories not contain a vertex is crucial, because one way we show that two trajectories are not in the
same cylinder is by showing that any cylinder containing both trajectories contains a vertex. This is especially useful in light of
the following standard results, the first of which can be found in Section 3 of \cite{Billiards}.

\begin{lem}\label{vertices}
For any $x\in P$, the set of angles $\theta\in S^1$ such that $\mathrm{Orb}(x,\theta)$ 
comes arbitrarily close to vertices of $P$ has full measure.
\end{lem}

Clearly any two periodic trajectories in the same cylinder will have the same combinatorial type. In fact, the converse
of this holds as well. The proof of this fact is obvious to those familiar with billiards, but is included here for completeness.

\begin{lem}\label{same combinatorial type}
If $f\sim g$, then up to reversal $f$ and $g$ lie in the same cylinder of periodic trajectories.
\end{lem}
\begin{proof}
Up to reversal, $f$ and $g$ strike the same sequence of edges.
They also must have the same initial angle, as otherwise by considering the unfolding it is clear that they would
eventually strike different edges. Let $h:[0,1]\times [0,1]\to P$ be the cylinder obtained by translating from $f$ to $g$. Clearly $h_s(1)=h_s(0)$ for all $s\in [0,1]$,
so the only possible obstruction to this being extendable to a cylinder of periodic trajectories is that its image contains a vertex of $P$. But then we would have some smallest $t$
such that there exists some $s\in [0,1]$ for which $h_s(t)$ is a vertex of $P$, and thus at this point $f(t)=h_0(t)$ and $g(t)=h_1(t)$ lie on different edges, contradicting the
fact that they strike the same sequence of edges.
\end{proof}

\begin{rem}
Since the copies of $P$ in $R(P)$ are indistinguishable from the edge labeling, if $f$ and $g$ are regarded as trajectories on $R(P)$ then in order for Lemma ~\ref{same combinatorial type}
to hold we must allow an action of $G(P)$. Similarly, it was necessary to allow reversal in the definition of combinatorial type since there is no canonical choice of
orientation for $R(P)$ (as some elements of $G(P)$ are orientation-reversing).
\end{rem}

The following lemmas provide a useful characterization of orbit tiles.
A version of this lemma for triangles is proven in Section 2.5 of \cite{Deg100}; however that version is more suited for computations
and less useful for our purposes.

\begin{lem}\label{tiles lemma}
Let $\mathcal P$ denote the set of polygons $P$ with any of the standard metrics.
Let $j_1,\cdots, j_k\in \{1,2,\ldots,n\}$ with $j_1=j_k$ and let $\theta_i$ denote the angles of $P$. Then there exist
\begin{itemize}
\item a continuous, analytic a.e. function $\Omega: \mathcal P\times (0,1)\times S^1\to \mathbb R$ such that $\mathrm{Orb}(P,xv_{j_1}+(1-x)v_{j_1+1},\theta)$
 strikes $e_{j_1},\ldots,e_{j_k}$ in order iff $\Omega(P,x,\theta)>0$,
\item an analytic function $D:\mathcal P\to S^1$ such that, assuming $\mathrm{Orb}(P,xv_{j_1}+(1-x)v_{j_1+1},\theta)$ strikes $e_{j_1},\ldots,e_{j_k}$,
 it is periodic iff $D(P)=\theta,$ and
\item if $k$ is even, a linear function $\Theta(\theta_1,\ldots,\theta_n)$ with even integer coefficients which gives the change in direction of any trajectory striking $e_{j_1},\ldots,e_{j_k}$.
\end{itemize}
\end{lem}
\begin{proof}
For $\Omega$, consider the unfolding of $P$ along the edges $e_{j_1},\ldots,e_{j_k}$. Note that for each $1\le i<k$, any trajectory which passes through $e_{j_1},\ldots,e_{j_i}$ passes
 through $e_{j_{i+1}}$ iff it passes between $v_{j_{i+1}}$ and $v_{j_{i+1}+1}$.
 The position of each vertex and the slope of each edge is an analytic function of $P$. 
 Thus the intersection of $\mathrm{Orb}(xv_{j_1}+(1-x)v_{j_1+1},\theta)$ with each edge $e_{j_i}$ is an analytic function of $P,x$ and $\theta$.
 Call these functions $h_i(P,x,\theta)$. Let $p_i:\R^2\to \R$ be the functional which takes the component of a vector in the direction of $e_{j_i}$. 
 We have shown that $\mathrm{Orb}(xv_{j_1}+(1-x)v_{j_1+1},\theta)$ is as desired iff $p_i(v_{j_{i}})<p_i(h_i(P,x,\theta))$ and $p_i(v_{j_{i}+1})>p_i(h_i(P,x,\theta))$ for all $i$. Thus if we let 
 \[\Omega(P,x,\theta)=\min\left(\min\{p_i(h_{i}(P,x,\theta)-v_{j_{i}})\},\max\{p_i(v_{j_{i}+1}-h_{i}(P,x,\theta))\}\right)\]
 we get the desired function.
 
For $D$, note that in order for the trajectory to return to its starting point, it must be parallel to the line connecting the first and last vertex described above. The direction of this line is easily seen
 to be an analytic function. 
 
For $\Theta$, note that upon striking an the edge $e_{j_i}$ and then $e_{j_{i+1}}$, the angle of a trajectory has been rotated by twice the angle between them, which is 
one of the $\theta_j$. Since $k$ is even, this gives us the desired function.
\end{proof}

For periodic trajectories, we can always assume $k$ is even as we can have the trajectory repeat.
We will use an easy but important fact for triangles that follows from Lemma ~\ref{tiles lemma}. This is stated in Section 2.1 of \cite{Deg100}, but not proven.
A less precise version is given by Lemma 7.1 in ~\cite{Deg1001}.

\begin{cor}\label{tiles classification}
An orbit tile of triangles is either an open set or an open subset of a line (in the subspace topology).
\end{cor}
\begin{proof}
Let $C$ be a combinatorial type, and note that $T\in \mathcal O(C)$ iff we have some $x\in (0,1)$ such that $\Omega(T,x,D(T))>0$ and $\Theta(T)=0$. Let
 $\Lambda_x(T)=\Omega(T,x,D(T))$. Thus 
 \[\mathcal O(C)=\ker \Theta\cap \left(\bigcup\limits_{x\in (0,1)}\Lambda_x^{-1}((0,\infty))\right)\]
and since $\Theta$ is affine of rank $0$ or $1$, $\ker\Theta$ can either be the entire space or a line. Since $\Lambda_x$ is continuous for all $x\in (0,1)$
 and $(0,\infty)$ is open, the union of the preimages is open and the result follows.
\end{proof}

Before proceeding to the proofs of our main theorems, we must point out two peculiarities of our definitions.

\begin{rem}
Some authors also consider combinatorial types equivalent if they differ by cyclic permutations, to allow for different parametrizations
of the trajectories.
We will also require that trajectories originate from the base of the triangle, which is justified by the fact that a periodic trajectory
must strike all three edges.
It should be clear from the proof of Theorem ~\ref{main theorem} that the theorem also holds with these two peculiarities removed.
However, we do not know whether Theorem ~\ref{converse} holds if we consider combinatorial types differing by cyclic permutations
equivalent. On the other hand, if we do not consider such combinatorial types equivalent, we must require that trajectories
originate on the base in order for Theorem ~\ref{main theorem} to hold. Fortunately these peculiarities
do not impact the corollaries.
\end{rem}


\section{Main Theorems}

From this point on, $T$ denotes an isosceles triangle and $T'$ the right triangle obtained by identifying the points in $T$ with their reflection across
 the line of symmetry. The labeling of vertices is shown in
 Figure ~\ref{labeling figure}.

We can reduce Theorem ~\ref{main theorem} to the rational case by the following lemma.

\begin{lem}\label{rational}
Any orbit tile containing an irrational isosceles triangle $T$ contains a rational isosceles triangle with base angle $\frac{a}{b}\pi$ where $b$ is even.
\end{lem}
\begin{proof}
Recall the function $\Theta$ from Lemma ~\ref{tiles lemma}. By the lemma, the zeros of $\Theta$ satisfy
$x\theta_1+y\theta_2+z\pi=0$ for some $x,y,z\in\mathbb Z$, and since $\theta_1=\theta_2=\theta$ we get $(x+y)\theta+z\pi=0$. If $x+y=0$ then $z=0$ as well. 
Thus either $\Theta$ is identically $0$, in which case the orbit tile is open, or $\ker \Theta$ is the line of isosceles
triangles so by Corollary ~\ref{tiles classification} the orbit tile is an open subset of this line. 
Since the set of isosceles triangles with base angle $\frac{a}{b}\pi$ where $b$ is even is dense in the line of isosceles triangles, the orbit tile must contain some such triangle.
If $x+y\ne 0$ then $\theta=\frac{-z}{x+y}\pi$ and the third angle is $\pi+\frac{2z}{x+y}\pi$, contradicting the fact that $T$ is irrational.
\end{proof}

\begin{figure}
\begin{tikzpicture}[scale=1.15]
\draw (0,0) -- node[midway, left]{2} (1.5,1.8) -- node[midway, right]{1} (3,0) -- node[midway, below]{3} (0,0);
\draw (.3,0) arc (0:30:5.6mm);
\draw (.5,.22) node{$\frac{a}{b}\pi$};
\draw (2.75,.3) arc (150:180:5.6mm);
\draw (2.46,.24) node{$\frac{a}{b}\pi$};
\draw (5,0) -- node[midway, left]{2} (5,1.8) -- node[midway, right]{1} (6.5,0) -- node[midway, below]{3} (5,0);
\draw (5,.2) -- (5.2,.2) -- (5.2,0);
\draw (6.25,.3) arc (150:180:5.6mm);
\draw (5.96,.24) node{$\frac{a}{b}\pi$};
\end{tikzpicture}
\caption{The triangles $T$ and $T'$.}
\label{labeling figure}
\end{figure}
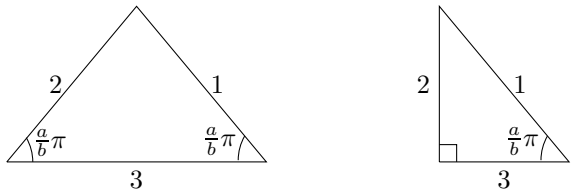

Note that Lemma ~\ref{rational} does not show that every periodic trajectory on an irrational isosceles triangle is stable. Indeed,
it is an easy exercise to show that the orbit tile of the combinatorial type $3132$ contains all isosceles triangles,
but no other triangles, hence is associated with an unstable trajectory.

\subsection{Theorem ~\ref{main theorem}}

In this section, we assume $b$ is even. The bulk of Theorem ~\ref{main theorem} is contained in the following lemma, which is illustrated by Figure 3.2.

\begin{lem}\label{double-cover}
There exists a double-cover $\pi:R(T)\to R(T')$ of Riemann surfaces such that $\dz_T=\dz_{T'}\circ \d\pi$.
Furthermore, the deck transformation $\lambda$ preserves combinatorial types.
\end{lem}
\begin{proof}
Note that $G(T)=G(T')$, since $r_1$ and $r_3$ agree in both groups while $r_3(r_1r_3)^{b/2}\in G(T)$ equals $r_2\in G(T')$.
Let $r$ denote these maps. Let $\vp_\alpha$ denote the inclusion map into $R(T)$ and $\vp_\alpha'$ the inclusion map into $R(T')$.
For $z$ in the image of $\vp_\alpha$, define $\pi:R(T)\to R(T')$ by
\[\pi(z)=\begin{cases}
\vp_{\alpha}'(\vp_{\alpha}^{-1}(z)) &\text{if } \Re(\alpha^{-1}\vp_{\alpha}^{-1}(z))\ge 0\\
\vp_{r\alpha}'(r\vp_{\alpha}^{-1}(z)) &\text{if } \Re(\alpha^{-1}\vp_{\alpha}^{-1}(z))\le 0
\end{cases}\]
A simple check shows that this is well-defined, locally biholomorphic and that $\dz_T=\dz_{T'}\circ \d\pi$.
Intuitively we are mapping each copy of $T$ in $R(T)$ onto a pair of copies of $T'$ in $R(T')$ identified along $e_2$.
Since each copy of $T$ is mapped onto two copies of $T'$, $\pi$ is a double-cover.

For $z$ in the image of $\vp_\alpha$, it is easy to verify that the deck transformation $\lambda$ is given by $\lambda(z)=\vp_{r\alpha}(r\vp_\alpha^{-1}(z))$.
Let $f$ be a periodic trajectory on $R(T)$. Clearly $\d\lambda$ is linear, thus $\lambda\circ f$ is also a periodic trajectory. In order to show that
$\lambda\circ f$ has the same combinatorial type as $f$, it suffices to show that $\lambda \circ f$ is the reversal of $(r_1r_3)^{b/2}\circ f$.
Note that $\lambda\circ f(0)= (r_1r_3)^{b/2}\circ f(0)$ as $f(0)\in e_3$. Since 
\[\d\lambda=\d r=-\d (r_1r_3)^{b/2}\]
we have $\d (\lambda \circ f)=-\d ((r_1r_3)^{b/2}\circ f)$, thus $\lambda\circ f(t)=(r_1r_3)^{b/2}\circ f(1-t)$.
\end{proof}
 
\begin{figure}
\begin{tikzpicture}
[scale=.73,midarrow/.style={thick,decoration={markings,mark=at position 0.5 with {\arrow{>}}},postaction={decorate}}]
    \foreach \x in {1,...,8}
    {   \ifthenelse{\isodd{\x}}
        {\xdef\mysign{1}}
        {\xdef\mysign{-1}}
        \fill[black!20!white,,opacity=0.4,shift={(0,\mysign*\octagonradius)}] (0,0) -- (45*\x+22.5:\octagonbigradius) -- (45*\x+67.5:\octagonbigradius) -- cycle;
        \fill[black!0!white,,opacity=0.4,shift={(0,\mysign*\octagonradius)}] (0,0) -- (45*\x+22.5:\octagonbigradius) -- (45*\x-22.5:\octagonbigradius) -- cycle;
        \draw[shift={(0,\mysign*\octagonradius)}] (0,0) -- (45*\x+22.5:\octagonbigradius);
        \draw[shift={(0,\mysign*\octagonradius)}] (0,0) -- (45*\x+67.5:\octagonbigradius);
    }
    \foreach \x in {0,...,6}
    {   \ifthenelse{\isodd{\x}}
        {\xdef\mysign{1}}
        {\xdef\mysign{-1}}
        \draw[midarrow,shift={(0,\mysign*\octagonradius)}] (45*\x-22.5:\octagonbigradius) -- (45*\x-67.5:\octagonbigradius);
        \draw[midarrow,shift={(0,-1*\mysign*\octagonradius)}] (45*\x+112.5:\octagonbigradius) -- (45*\x+157.5:\octagonbigradius);
    }
    \draw[shift={(0,\octagonradius)}] (247.5:\octagonbigradius) -- (292.5:\octagonbigradius);
    \draw[thick,->] (-\octagonbigradius,\octagonbigradius) to[bend right=15] node[left] {$\lambda$} (-\octagonbigradius,-\octagonbigradius);
    \draw[thick,->] (\octagonbigradius,-\octagonbigradius) to[bend right=15] node[right] {$\lambda$} (\octagonbigradius,\octagonbigradius);
    \draw[thick,->] (1.05*\octagonbigradius, 1.05*\octagonbigradius) -- node[label=90:$\pi$] {}  (3.44*\octagonbigradius, .24*\octagonbigradius);
    \draw[thick,->] (1.05*\octagonbigradius, -1.05*\octagonbigradius) -- node[label=270:$\pi$] {}  (3.44*\octagonbigradius, -.25*\octagonbigradius);
    \foreach \x in {1,...,8}
    {
        \fill[black!20!white,,opacity=0.4,shift={(3.02*\octagonbigradius,.9*\octagonbigradius+\mysign*\octagonradius)}] (0,0) -- (45*\x+22.5:\octagonbigradius) -- (45*\x+67.5:\octagonbigradius) -- cycle;
        \draw[shift={(3.02*\octagonbigradius,.9*\octagonbigradius+\mysign*\octagonradius)}] (0,0) -- (45*\x+22.5:\octagonbigradius);
        \draw[shift={(3.02*\octagonbigradius,.9*\octagonbigradius+\mysign*\octagonradius)}] (0,0) -- (45*\x:.924*\octagonbigradius);
     }
    \foreach \x in {0,...,7}
    {   
        \draw[midarrow,shift={(3.02*\octagonbigradius,-.025*\octagonbigradius)}] (45*\x:.922*\octagonbigradius) -- (45*\x+22.5:\octagonbigradius);
        \draw[midarrow,shift={(3.02*\octagonbigradius,-.025*\octagonbigradius)}] (45*\x:.922*\octagonbigradius) -- (45*\x-22.5:\octagonbigradius);
    }
\end{tikzpicture}
\caption{The surfaces $R(T)$ and $R(T')$ for a base angle of $\frac38\pi$.
The arrows indicate edge orientations. Parallel edges are identified iff they share the same orientation. 
The covering $\pi$ and deck transformation $\lambda$ are illustrated. Labels are not shown.}
\label{covering figure}
\end{figure}
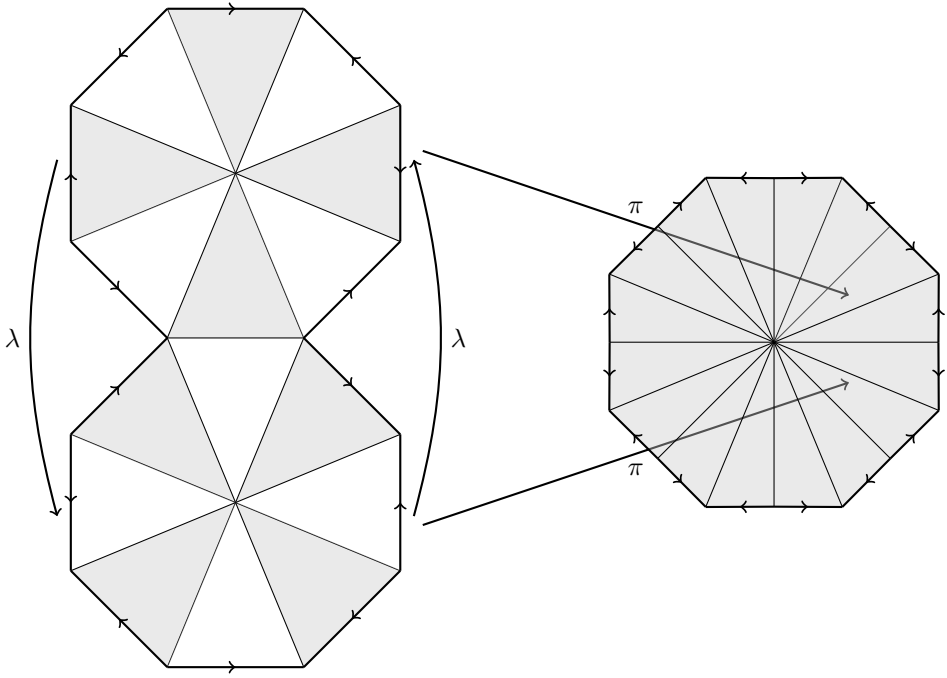

We are now able to prove Theorem ~\ref{main theorem} and Corollary ~\ref{corollary}.

\begin{proof}[Proof of Theorem ~\ref{main theorem}]
Note that the trajectories $f'$ and $g'$ on $T'$ obtained from $f$ and $g$ are $\pi\circ f$ and $\pi\circ g$. Since these have the same combinatorial type, by Lemma ~\ref{same combinatorial type}
up to reversal we have some $\alpha\in G(T')$ and some
cylinder of periodic trajectories $h:(0,1)\times [0,1]\to R(T')$ such that for some $u,v\in (0,1)$, $\pi\circ f=h_u$ and $\alpha\circ \pi\circ g=h_v$. Lifting this to
a function $\tilde{h}:(0,1)\times [0,1]\to R(T)$, we get that $\tilde{h}_s$ has constant argument for each $s\in (0,1)$ since $\dz_T=\dz_{T'}\circ \d\pi$. Since $\tilde{h}$ can be chosen such that
$\tilde{h}_u=f$, with this choice of $\tilde{h}$ we see that $\tilde{h}_u(0)=\tilde{h}_u(1)$, thus $\tilde{h}_s(0)=\tilde{h}_s(1)$ for all $s\in (0,1)$ as the lift of
the line $h_s(1):(0,1)\to R(T')$ is determined by $\tilde{h}_u(1)$, thus $\tilde{h}$ is a cylinder of periodic trajectories. It is easy to see that 
$\alpha\circ \pi = \pi \circ \alpha$, thus we get that $\tilde{h}_v$ is a lift of $\alpha\circ g$.
Since $\lambda$ preserves combinatorial types, it follows that $\tilde{h}_v$ has the same combinatorial type as $\alpha\circ g$.
Since $f$ and $\tilde{h}_v$ lie in the same cylinder of periodic trajectories, $f\sim \tilde{h}_v$.
Thus $f\sim g$ as well. 
\end{proof}

\begin{proof}[Proof of Corollary ~\ref{corollary}]
Let $f$ be a periodic trajectory on an isosceles triangle $T$ satisfying Condition ~\ref{condition}. 
Let $g$ be the trajectory obtained by reflecting $f$ across the line of symmetry of $T$. Since the roles of the two congruent edges
are reversed, it is clear that the orbit tile of $g$ is the reflection of that of $f$ across the line of isosceles triangles. But clearly $f'\sim  g'$, thus
by Theorem ~\ref{main theorem} $f\sim g$, hence the orbit tiles of $f$ and $g$ agree. It follows that the orbit tile of $f$ is symmetric across the line of isosceles triangles.
\end{proof}

\subsection{Theorem ~\ref{converse}}

We now assume $b$ is odd. In this case, instead of a double cover of $R(T')$ by $R(T)$, we get a biholomorphism, as shown in Figure 3.3. This is a result of the fact (which can be easily verified) that when $b$ is odd,
$r_2\in G(T')$ does not correspond to any element of $G(T)$, so $G(T)$ is a proper subgroup of $G(T')$.
Note that the under our embedding of $T$, the requirement that $x$ not be the center of the base is equivalent to $x\ne 0$.

\begin{lem}\label{biholomorphism}
$R(T)$ and $R(T')$ are biholomorphic, and the $1$-forms $\dz_T$ and $\dz_{T'}$ agree under this biholomorphism.
\end{lem}

\begin{proof}
It is easy to see that the map $\pi$ defined in Lemma ~\ref{double-cover} is also a covering in this case. Since $G(T)$ is a proper subgroup of $G(T')$,
we can see that in this case $\pi$ is injective, thus a biholomorphism. Since $\dz_T=\dz_{T'}\circ \d\pi$ by the lemma, the $1$-forms agree under $\pi$.
\end{proof}

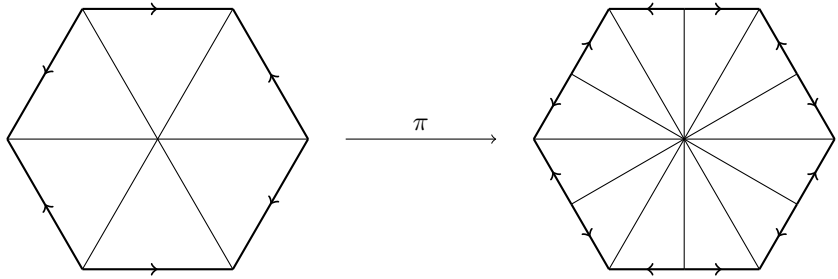
\begin{figure}
\begin{tikzpicture}
[scale=2,midarrow/.style={thick,decoration={markings,mark=at position 0.5 with {\arrow{>}}},postaction={decorate}}]
\draw[midarrow] (1,0) -- (.5,.866);
\draw[midarrow] (-.5,.866) -- (.5,.866);
\draw[midarrow] (-.5,.866) -- (-1,0);
\draw[midarrow] (-.5,-.866) -- (-1,0);
\draw[midarrow] (-.5,-.866) -- (.5,-.866);
\draw[midarrow] (1,0) -- (.5,-.866);
\draw (-1,0) -- (1,0);
\draw (-.5,-.866) -- (.5,.866);
\draw (.5,-.866) -- (-.5,.866);
\draw[midarrow] (3.5+.75,.433) -- (3.5+1,0);
\draw[midarrow] (3.5+.75,.433) -- (3.5+.5,.866);
\draw[midarrow] (3.5+0,.866) -- (3.5+.5,.866);
\draw[midarrow] (3.5+0,.866) -- (3.5+-.5,.866);
\draw[midarrow] (3.5+-.75,.433) -- (3.5+-.5,.866);
\draw[midarrow] (3.5+-.75,.433) -- (3.5+-1,0);
\draw[midarrow] (3.5+.75,-.433) -- (3.5+1,0);
\draw[midarrow] (3.5+.75,-.433) -- (3.5+.5,-.866);
\draw[midarrow] (3.5+0,-.866) -- (3.5+.5,-.866);
\draw[midarrow] (3.5+0,-.866) -- (3.5+-.5,-.866);
\draw[midarrow] (3.5+-.75,-.433) -- (3.5+-.5,-.866);
\draw[midarrow] (3.5+-.75,-.433) -- (3.5+-1,0);
\draw (3.5+-1,0) -- (3.5+1,0);
\draw (3.5+-.5,-.866) -- (3.5+.5,.866);
\draw (3.5+.5,-.866) -- (3.5+-.5,.866);
\draw (3.5+0,-.866) -- (3.5+0,.866);
\draw (3.5+.75,-.433) -- (3.5+-.75,.433);
\draw (3.5+-.75,-.433) -- (3.5+.75,.433);
\draw[->] (1.25,0) to node[above] {$\pi$} (2.25,0);
\end{tikzpicture}
\caption{The surfaces $R(T)$ and $R(T')$ for a base angle of $\frac{1}{3}\pi$.
The arrows indicate edge orientations. Parallel edges are identified iff they share the same orientation. Labels are not shown.}
\label{biholomorphism figure}
\end{figure}

This biholomorphism allows us to interpret the action of $G(T')$ on $R(T')$ as an action on $R(T)$. In particular, $r_2\in G(T')$ acts distinctly from any element of $G(T)$.
This gives us a way to construct trajectories violating Theorem ~\ref{main theorem}.
Combining this with Lemma ~\ref{vertices}, we are able to prove Theorem ~\ref{converse} and Corollary ~\ref{unstable}. 

\begin{proof}[Proof of Theorem ~\ref{converse}]
Let $f=\mathrm{Orb}(x,\theta)$ and assume $f$ is periodic. Suppose $f$ satisfies Theorem ~\ref{main theorem}.
Then $f\sim r_2\circ f$. Thus we have some $\alpha\in G(T)$ such that
$f$ and $\alpha\circ r_2\circ f$ lie in the same cylinder of periodic trajectories, up to reversal. Hence their arguments must
agree up to negation, so $\mathrm{d}f$ lies in the $\pm1$-eigenspace of $\d(\alpha\circ r_2)=\alpha r_2$. Since this cannot be the
identity, either it is a reflection across a line of angle a multiple of $\frac{a}{b}\pi$ or it is a nontrivial rotation. In the first
case, $\theta$ is of the form $n\frac{a}{b}\pi$ or $\frac{\pi}{2}+n\frac{a}{b}\pi$, and the set of such angles is finite.

Let $(a,b)$ be an interval in $S^1$. By Lemma ~\ref{vertices}, we have some vertex $v$ of a copy of $T$ in $R(T)$ lying along a trajectory passing through $0$
with argument $\phi\in (a,b)$ which is not of the form $n\frac{a}{b}\pi$ or $\frac{\pi}{2}+n\frac{a}{b}\pi$.
Suppose $\theta=\phi+\epsilon$,
with $|\epsilon|$ sufficiently small that $\phi+\epsilon$ is not of the form $n\frac{a}{b}\pi$ or $\frac{\pi}{2}+n\frac{a}{b}\pi$.
Thus $\alpha r_2$ must be a nontrivial rotation with real eigenvalue, hence is a rotation by $\pi$, so $\alpha=r_3$ and
$f$ lies in the same cylinder of periodic trajectories as the reversal of $r_3r_2\circ f$.
Let $h$ be the cylinder containing $f$ and the reversal of $r_3r_2\circ f$,
say with $h_u(t)=f(t)$ and $h_v(t)=r_3r_2\circ f(1-t)$. Then $h_s(0)$ lies in the $e_3$ and connects $x$ and
$r_2x$, thus $h$ has width at least $|x-r_2x|$ which is positive since $x\ne 0$.
Since the cylinder varies continuously with $\epsilon$, if $|\epsilon|$ is sufficiently small the cylinder contains the trajectory
connecting $0$ and $v$, thus contains $v$, a contradiction.
Hence the set of such $\theta$ is not dense in $(a,b)$, so this set is nowhere dense.
\end{proof}

\begin{proof}[Proof of Corollary ~\ref{unstable}]
It suffices to note that Theorem ~\ref{main theorem} is a statement about combinatorial types, which are constant on orbit tiles.
Thus if $f$ is a stable trajectory originating from $x$ with argument $\theta$, 
Theorem ~\ref{main theorem} holds for $f$ hence $\theta$ lies in the nowhere dense set from Theorem ~\ref{converse}.
\end{proof}

The same line of reasoning shows that Theorem\;\ref{converse} provides a converse to Corollary\;\ref{corollary} as well. If the orbit tile of a trajectory were equal to its reflection
across the line of isosceles triangles, then by Lemma\;\ref{tiles classification} either it is an open set or an open subset of the line of isosceles triangles. In either case, the
tile contains an isosceles triangle satisfying Condition\;\ref{condition}.

The assumption that $x\ne 0$ is crucial for Theorem ~\ref{converse}, as it is easy to see that when $x=0$, $r_2\circ f$ is the reversal
of $f$ and thus Theorem ~\ref{main theorem} holds for $f$. It is unclear whether it is necessary for Corollary ~\ref{unstable}; in fact it
is possible that Corollary ~\ref{unstable} holds for all polygons.

We close with a remark regarding further applications of these results.

\begin{rem}
Let $T$ be an isosceles triangle with base angle $\frac{a}{b}\pi$ where $b$ is odd.
The proof of Theorem ~\ref{converse}, together with Corollary ~\ref{corollary}, gives a classification of the stable trajectories on $R(T)$. Either they have
argument $0$ or $\pi/2$ up to the action of $G(T)$, or they lie in the same cylinder as some trajectory passing through the center
of the base of $T$. These latter trajectories are precisely the ``mirror trajectories'' first defined and studied by
Galperin and Zvonkine in \cite{mirror}. This classification may be of use in attacking open problems concerning the behavior of stable
trajectories on isosceles triangles, such as the conjecture in \cite{NearlyIsosc} that no finite collection of orbit tiles covers
any neighborhood of certain isosceles triangles with the property of Veech.
\end{rem}

\subsection*{Acknowledgments}
It is a pleasure to thank Peter May for his sponsorship through the University of Chicago's 
REU program, which supported my initial investigation, and my mentor Ilya Gekhtman for his guidance under this program. 
I would like to thank Pat Hooper and Richard Schwartz for developing the 
McBilliards software, which first suggested Corollary ~\ref{corollary} and led to my investigation of isosceles billiards.
In addition, I would like to thank Howard Masur and Richard Schwartz for their helpful advice and comments on a version of this manuscript.


\begin{thebibliography}{9}

\bibitem{Billiards}
Boldrighini, C., Keane, M. and Marchetti, F. ``Billiards in Polygons,''
The Annals of Probability, {\bf 6} (1978), No. 4, 532-540.
%

\bibitem{mirror}
Galperin, G. and Zvonkine, D., ``Periodic Billiard Trajectories in Right Triangles and Right-Angled Tetrahedra,''
Regular \& Chaotic Dynamics, {\bf 8} (2003), No. 1, 29-44.

\bibitem{Survey}
Gutkin, E., ``Billiards in Polygons: Survey of Recent Results,''
Journal of Statistical Physics, {\bf 83} (1996), 7-26.

\bibitem{RightTri}
Hooper, W. P., ``Periodic billiard paths in right triangles are unstable,''
Geometriae Dedicata, {\bf 125} (2007), 39-46. 

\bibitem{NearlyIsosc}
Hooper, W. P. and Schwartz, R. E., ``Billiards in nearly isosceles triangles,''
J. Mod. Dyn., {\bf 3} (2009), 159-231.

\bibitem{RatCase}
Masur, H., ``Closed trajectories for quadratic differentials with an application to billiards,''
Duke Math J., {\bf 53} (1986), 307-314.

\bibitem{FlatStruc}
Masur, H. and Tabachnikov, S., ``Rational billiards and flat structures,''
Handbook of Dynamical Systems, North-Holland, Amsterdam, {\bf 1A} (2002), 1015-1089.

\bibitem{Deg1001}
Schwartz, R. E., ``Obtuse triangular billiards I: near the (2, 3, 6) triangle,''
Experimental Mathematics, {\bf 15} (2006), No. 2, 161-182.

\bibitem{Deg100}
Schwartz, R. E., ``Obtuse Triangular Billiards II: 100 Degrees Worth of Periodic Trajectories,''
Experimental Mathematics, {\bf 18} (2008), No. 2, 137-171.


\end{thebibliography}
\end{document}